\documentclass[twoside]{article}
\usepackage{subfig}
\usepackage{amsmath,amsthm, amssymb, wasysym}
\usepackage{amsfonts}
\usepackage{graphicx}
\usepackage{verbatim}
\usepackage{enumerate}
\usepackage{color}
\usepackage[colorlinks,citecolor=blue,urlcolor=blue]{hyperref}
\usepackage{chngcntr}
\usepackage{hyperref}
\usepackage{placeins}

\usepackage{csquotes}
\usepackage[backend=bibtex, style=authoryear, maxnames=5, maxcitenames=3]{biblatex}

\newtheorem{thm}{Theorem}[section]
\newtheorem{cor}[thm]{Corollary}

\theoremstyle{definition}
\newtheorem{defn}[thm]{Definition}

\theoremstyle{remark}

\newtheorem{exmpl}[thm]{Example}

\newcommand{\norm}[1]{\left\Vert#1\right\Vert}

\newcommand{\abs}[1]{\left\vert#1\right\vert}

\newcommand{\R}{\mathbb R}

\newcommand{\To}{\rightarrow}

\DeclareMathOperator{\Span}{span}

\DeclareMathOperator{\Cone}{cone}

\DeclareMathOperator{\qint}{qint}

\newcommand\blfootnote[1]{%
  \begingroup
  \renewcommand\thefootnote{}\footnote{#1}%
  \addtocounter{footnote}{-1}%
  \endgroup
}

\title{Proper Scoring Rules and Bregman Divergences}
\author{Evgeni Y. Ovcharov
\thanks{The author has been supported by the European Union Seventh Framework Programme under grant agreement no. 290976}
\\
\small Heidelberg Institute for Theoretical Studies\\
\small  Schloss-Wolfsbrunnenweg 35, D-69118 Heidelberg, Germany\\
\small \texttt{trulr6@yahoo.com}
}

\date{\today}

\begin{document}
\maketitle

\pagestyle{myheadings}
\markboth{Ovcharov}{Proper Scoring Rules and Bregman Divergences}

\begin{abstract}%
We revisit the mathematical foundations of proper scoring rules (PSRs) and Bregman divergences and present their characteristic properties in a unified theoretical framework. In many situations it is preferable not to generate a PSR directly from its convex entropy on the unit simplex but instead by the sublinear extension of the entropy to the positive orthant. This gives the scoring rule simply as a subgradient of the extended entropy, allowing for a more elegant theory. The other convex extensions of the entropy generate affine extensions of the scoring rule and induce the class of functional Bregman divergences. We discuss the geometric nature of the relationship between PSRs and Bregman divergences and extend and unify existing partial results. We also approach the topic of differentiability of entropy functions. Not all entropies of interest possess functional derivatives, but they do all have directional derivatives in almost every direction. Relying on the notion of quasi-interior of a convex set to quantify the latter property, we formalise under what conditions a PSR may be considered to be uniquely determined from its entropy.
\end{abstract}



\section{Introduction}

Proper scoring rules (PSRs) originated in probabilistic forecasting as devices that assess the quality of forecasts and elicit private information. The subject enjoys a considerable applied and theoretical interest in recent years \parencite{GK}. The present paper focuses on mathematical and geometric aspects of PSRs and Bregman divergences and elucidates the relationship between them. Having evolved to a large degree separately, the two notions have been investigated under restrictive and specialised conditions. We survey the available literature on this topic and systematise the relevant results by presenting them in a general and unified theoretical framework. A more detailed discussion on individual aspects of our review is given in the subsection below.

First, let us outline how the rest of the paper is organised. In Section \ref{sect: char}, we discuss the characterisation of PSRs and the related canonical extension of the entropy as a sublinear function to the positive orthant. In Section \ref{sect: conv}, we explore more general convex extensions of the entropy function to and beyond the positive orthant. This construction generates affine extensions of PSRs, also known as \emph{affine scoring rules}, and induces the class of functional Bregman divergences. In Section \ref{sect: Bregman}, we examine and generalise some technical results about Bregman divergences under regularity conditions that are natural for PSRs. We investigate the differentiability properties of entropy functions in Section \ref{sect: diff}. Here, we describe the collection of all PSRs generated by a given entropy function and formalise under what conditions this collection contains a unique element. In the short Appendix,  we present the proof of a technical result.

\subsection{Motivation and relation to literature}

The characterisation of PSRs through the 1-homogeneous extension of the entropy to the positive orthant was first developed by \cite{McC, HB}. A simpler characterisation of PSRs is due to  \cite{Sav, GR}, who consider entropy functions on the unit simplex. The unit simplex is, however, a negligible set in measure and topology, which obfuscates questions pertaining to regularity and uniqueness of subgradients, differentiability of entropy functions, etc. On the other hand, any proper scoring rule on the unit simplex is simply a subgradient relative to the positive orthant of the 1-homogeneous extension of the entropy. This fact provides us with means not only to study regularity of entropy functions, but also it offers a precise geometric interpretation for the condition of propriety of a scoring rule. The extension is implicit in the context of scoring rules that are 0-homogeneous in form such as the \emph{proper local scoring rules of higher orders} \parencite{DLP,PDL}. A prominent member of that class is the \emph{Hyv\"arinen scoring rule}, which supplies an attractive and statistically consistent alternative for the method of \emph{pseudolikelihood} \parencite{DawMus1, DawMus, Hyv, Hyv1, ForbLau}. The \emph{pseudospherical scoring rules} \parencite{GR, D} are another important family of 0-homogeneous scoring rules.

Confining attention only to sublinear extensions of the entropy instead of the more general convex extensions is too limiting. For example, the simplest convex extension of the \emph{power entropy}, corresponding to the \emph{power scoring rules}, is the power function. This family is very popular in the meteorological literature, mainly in terms of the \emph{quadratic scoring rule}, or the analogous \emph{Brier score} \parencite{B, GK}. The corresponding entropy is also known as \emph{Tsalis entropy}, a concept that originates in the physics literature \parencite{DawMus1}. The power scoring rules are familiar for their robustness properties both under infinitesimal contamination \parencite{BHHJ} and heavy contamination \parencite{KF}. The latter work provides some practical justification for our interest in extending PSRs to the positive orthant and beyond, as its methods rely on unnormalised statistical models.

General convex extensions of the entropy naturally appear in the context of Baysian games, where the analogous quantities to scoring rules are termed \emph{allocation rules} \parencite{FK}. In this broader context the authors introduce the notion of an \emph{affine score}, which may be visualised geometrically as a family of supporting hyperplanes to a convex function. This construction generalises the expected scores of PSRs and induces the class of functional Bregman divergences. The same structure may be found in the elicitation of expectiles, and other linear functionals of predictive densities, because the associated consistent scoring rules have the form of a Bregman divergence \parencite{AF}. Convexity plays an important role in more general elicitation problems \parencite{Steinw, Zieg, Will}.

Bregman divergences are central objects in machine learning and statistics where they serve as natural generalisations to the Euclidean metric. Their properties have been deeply studied on Euclidean spaces \parencite{Baner, BauBor, Boi}, and partial generalisations are available in the context of functional spaces \parencite{Frig}. The latter work, however, uses assumptions that are not general enough to include most of the proper scoring rules of practical interest. In contrast, here we present both notions under unified regularity conditions and demonstrate that the characterisation of Bregman divergences in the Euclidean setting \parencite{Baner} extends to the present setting. Another aspect we investigate here is the well-known fact that the generalised quadratic divergence is the only symmetric Bregman divergence on Euclidean spaces \parencite{Boi}. We find an analogue of this fact in the context of a very general class of functional Bregman divergences.

It is interesting to understand in what formal sense an entropy function defines a unique PSR. In finite dimensions, or if the entropy function allows a continuous extension to an open cone in a normed space, the question may be resolved with the standard methods of convex analysis \parencite{O}. Specifically, the entropy function has a unique subgradient at an interior point of its domain if and only if it is differentiable at that point \parencite{BVff}. In infinite dimensions, however, things get complicated due to the fact that many standard function spaces, such as the Lebesgue spaces over $\R^n$, are endowed with positive orthant that has empty interior. This implies that any extension of the positive orthant to an open cone will contain densities that change sign. The entropies of many important scoring rules, such as the logarithmic scoring rule and the proper local scoring rules of higher orders, cannot be defined for signed densities. Consequently, these entropies are not differentiable in the standard sense. It turns out that we may still resolve our question with the help of the notion of \emph{quasi-interior}. The latter notion refines the notion  of interior of a convex set in infinite dimensions when the interior is empty. For our purposes, we need the algebraic version of quasi-interior from \cite{O}, which is analogous to its better-known topological equivalents \parencite{BorLew, FullBraun}. One of our key results there is the fact that an entropy function may still have a unique subgradient on the nonempty quasi-interior of a positive cone. As an illustration, we explicitly construct a positive cone with nonempty quasi-interior where the Hyv\"arinen scorng rule is the unique 0-homogeneous subgradient of its entropy function. Here, we discuss in greater detail some of the basic properties of algebraic quasi-interior and generalise the uniqueness result to an arbitrary convex domain.

\section{The canonical extension}\label{sect: char}

The common application of unnormalised statistical models in the literature motivates us to consider the possible extensions of PSRs to positives cones. The extension of the entropy function as a sublinear function to the positive cone of the set of probabilities is referred to as canonical. This extension is universal to all entropy functions and encapsulates the condition for propriety of a scoring rule directly, as we will see below.

We begin with some standard definitions. We fix a measure space $(\Omega, \mathcal A, \mu)$ and a convex class $\mathcal P$ of probability distributions on $\Omega$ which are absolutely continuous with respect to the measure $\mu$ and represented by their probability densities.

\begin{defn}
 We call the functions $f:\Omega\To\R$ \emph{$\mathcal P$-integrable} if
\begin{equation*}
   \int_{\Omega} \abs{f(x)}p(x)d\mu(x)<\infty
\end{equation*}
for every $p\in\mathcal P$. We denote by ${\mathcal L(\mathcal P)}$ the linear space of $\mathcal P$-integrable functions.
\end{defn}

Formally, any mapping $S:\mathcal P\To\mathcal L(\mathcal P)$ is a scoring rule. Suppose that $X$ is a random variable taking values in $\Omega$ with unknown true distribution $p\in\mathcal P$. If $q\in\mathcal P$ is a predictive density for $p$, then the random variable $S(q)(X)$ assigns a numerical score to each outcome of $X$. The assumption of $\mathcal P$-integrability ensures that $S(q)(X)$ has a finite expectation,
\begin{equation*}
  p\cdot S(q) := \int_{ \Omega} S(q)(x)p(x)d\mu(x),
\end{equation*}
which is also termed the \emph{expected score} of $S$. 
Viewing scoring rules as positive incentives which a forecaster wishes to maximise in the long run, only those scoring rules which satisfy the following condition encourage honesty.

\begin{defn}\label{defn: proper}
A scoring rule $S$ that maximises its expected score at the true density,
\begin{equation}\label{eq: proper}
  p\cdot S(p)  = \max_{q\in\mathcal P} p\cdot S(q),
\end{equation}
is called \emph{proper}. If the true density is always a unique maximiser, $S$ is called \emph{strictly proper}.
\end{defn}

Related concepts are the \emph{(negative) entropy},
\begin{equation}\label{eq: Euler 1}
  \Phi(p)=p\cdot S(p),
\end{equation}
for every $p\in\mathcal P$, and the \emph{score divergence} $D:\mathcal P\times\mathcal P\To\R$ given by
\begin{equation}\label{eq: Div S}
  D(p,q) = p\cdot S(p) - p\cdot S(q).
\end{equation}
It follows immediately from Definition \ref{defn: proper} that $\Phi$ is convex, being a pointwise maximum of linear functions, and that $D$ is nonnegative. Strict propriety is equivalent to $\Phi$ being strictly convex and to $D$ being positive-definite, i.e. equal to zero only for $p=q$.

Mathematically, propriety of a scoring rule is equivalent to convexity of the associated entropy, which will be the key structural property we explore in what follows.  All subsequent results about PSRs and Bregman divergences will be presented in the unified framework of the space $\Span\mathcal P$ (the linear span of $\mathcal P$) and its dual $\mathcal L(\mathcal P)$. Notice that in finite dimensions $\Span\mathcal P$ may be identified with some Euclidean space $\R^n$, and due to the fact the latter is self-dual, $\mathcal L(\mathcal P)$ also identifies with $\R^n$. In infinite dimensions, however, self-duality holds only in special cases and in general $\mathcal L(\mathcal P)$ is more refined than the algebraic dual of $\Span\mathcal P$, but less refined than the topological dual of $\Span\mathcal P$, whenever the latter is equipped with topology. Consequentially, the linear functionals in $\mathcal L(\mathcal P)$ are generally not continuous, and we will be primarily focused on their algebraic properties.

Throughout, by $\mathcal K$ we denote a convex set such that $\mathcal P\subset \mathcal K\subset\Span \mathcal P$. Therefore, the elements of $\mathcal K$ are linear combinations of probability densities.

\begin{defn}\label{defn: subgr 1}
Given a function $\Phi:\mathcal K\To\R$ and a point $q\in\mathcal K$, we say that $q^*\in\mathcal L(\mathcal P)$ is a \emph{($\mathcal P$-integrable) subgradient} of $\Phi$ at $q$ relative to $\mathcal K$ if
\begin{equation}\label{eq: subgr 1}
  \Phi(p) \geq (p-q) \cdot q^* + \Phi(q)
\end{equation}
for all $p\in\mathcal K$. If the above inequality is strict for all $p\not=q$, the subgradient $q^*$ is called \emph{strict}.
\end{defn}

So, subgradients are linear functionals that define supporting hyperplanes to the graph of a convex function. Specifically, the set
\[
  \{(p,y)\, |\,p\in\Span \mathcal P,\, y = (p-q) \cdot q^* + \Phi(q)\}
\]
is a \emph{supporting hyperplane} to $\Phi$ at $q$. A convex function may have many subgradients at a given point. The collection of all subgradients of $\Phi$ at $q$ is called the \emph{subdifferential} of $\Phi$ at $q$ and denoted by $\partial\Phi(q)$. Suppose that $\partial\Phi(q)\not=\emptyset$ for each $q\in\mathcal K$. Then, we call a selection of subgradients $S(q)\in\mathcal \partial\Phi(q)$, for each $q\in\mathcal K$, a \emph{subgradient} of $\Phi$ on $\mathcal K$.

Definition \ref{defn: subgr 1} implies the following characterisation of PSRs due to \cite{GR}.

\begin{thm}\label{thm: GR}
A scoring rule $S:\mathcal P\To\mathcal L(P)$ is (strictly) proper if and only if there exists a pair $(\Phi,\Phi^*)$, where $\Phi:\mathcal P\To\R$ is (strictly) convex and $\Phi^*:\mathcal P\To\mathcal L(\mathcal P)$ is a subgradient of $\Phi$ relative to $\mathcal P$, such that
\begin{equation}\label{eq: GR}
  S(q)(x) = \Phi^*(q)(x) + \Phi(q) - q\cdot\Phi^*(q),
\end{equation}
for every $q\in\mathcal P$.
\end{thm}

The merit of this result lies in the simplicity of its proof and the fact that it can be easily extended to arbitrary convex domains, as we will see in the next section. On the other hand, the theorem does not explain why only certain subgradients of $\Phi$ may be identified with PSRs, which means that we are still lacking a precise geometric interpretation of the condition of propriety.

Our next goal is to present such an interpretation by exploiting a beautiful connection with Euler's homogeneous function theorem. To that end, let us first review some properties related to homogeneity. For two sets $A$ and $B$ in $\Span\mathcal P$, we employ the Minkowski sum and difference notation: $A \pm B=\{a \pm b\,|\,a\in A\,, b\in B\}$. For $\lambda\in\R$ and $A\subset\Span\mathcal P$, we write $\lambda A=\{\lambda a\,|\,a\in A\}$. A set $C\subset\Span \mathcal P$ is called a \emph{convex cone} if $\lambda C=C$ and $C+C=C$ for all $\lambda>0$. Throughout, we take the conical hull of a set $C$, denoted $\Cone C$, to mean the smallest convex cone that contains $C$. Let a function $f: C\To\R$ be given, where $C$ is a convex cone. It is said that $f$ is \emph{$\alpha$-homogeneous} for some $\alpha\in\R$ if $f(\lambda q)=\lambda^\alpha f(q)$ for every $q\in C$ and every $\lambda>0$. Notice that a convex, 1-homogeneous function is a sublinear function. An extended version of \emph{Euler's homogeneous function theorem} states that if $\Phi: C\To\R$ is 1-homogeneous, then
\begin{equation}\label{eq: Euler 2}
  q\cdot\partial\Phi(q)=\Phi(q)
\end{equation}
for every $q\in C$ \parencite{HB, O}. The above identity relates sets, since $\partial\Phi(q)$ is generally a multi-valued map. It can be shown further that the subdifferential is a 0-homogeneous multi-valued map in the sense that it satisfies the relation $\partial\Phi(\lambda q)=\partial\Phi(q)$, for every $\lambda>0$ and every $q\in C$.

In view of the above, the extension of a PSR and its entropy as a 0-homogeneous and  1-homogeneous function, respectively, to $\Cone\mathcal P$ = $\{\lambda p\,|\, \lambda>0, p\in\mathcal P\}$ behaves consistently. Explicitly, given $S:\mathcal P\To\mathcal L(\mathcal P)$, we set
\begin{equation*}
  S(q) = S\left(\frac q {q\cdot 1}\right),
\end{equation*}
for every $q\in\Cone\mathcal P$, where $q\cdot 1$ is the \emph{normalising constant} of $q$. Similarly, for $\Phi:\mathcal P\To\R$, we write
\begin{equation*}
  \Phi(q) = (q\cdot 1) \Phi \left(\frac q {q\cdot 1}\right),
\end{equation*}
for every $q\in\Cone\mathcal P$. Due to \eqref{eq: Euler 2}, in the context of 1-homogeneous functions, Definition \ref{defn: subgr 1} reduces to the following.

\begin{defn}\label{defn: subgr 2}
Given a 1-homogeneous function $\Phi:\Cone \mathcal P\To\R$ and a point $q\in\Cone \mathcal P$, we say that $q^*\in\mathcal L(\mathcal P)$ is a \emph{($\mathcal P$-integrable) subgradient} of $\Phi$ at $q$ relative to $\Cone \mathcal P$ if
\begin{equation}\label{eq: subgrad}
  \Phi(p) \geq p \cdot q^*
\end{equation}
for all $p\in\Cone \mathcal P$, with equality for $p=q$. If the above inequality is strict for all $p$ not positively collinear to $q$, the subgradient $q^*$ is called \emph{strict}.
\end{defn}

Notice the special convention for a strict subgradient on $\Cone\mathcal P$ in the above definition. This notion of subgradient is coherent with the condition for propriety, which follows from the formal equivalence of Definition \ref{defn: proper} and Definition \ref{defn: subgr 2}. Thus we arrive at the classical characterisation of PSRs due to \cite{McC} and \cite{HB}. The formulation we give below emphasises the geometric nature of the result.

\begin{thm}\label{thm: HB}
Let $S:\mathcal P\To\mathcal L(P)$ be a scoring rule and $\Phi:\mathcal P\To\R$ be defined as $\Phi(p)=p\cdot S(p)$, for every $p\in\mathcal P$. Then $S$ is (strictly) proper if and only if the 0-homogeneous extension of $S$ to $\Cone \mathcal P$ is a (strict) subgradient of the 1-homogeneous extension of $\Phi$ to $\Cone \mathcal P$.
\end{thm}

See also \cite{Will} who characterises PSRs by making use of the duality theory of convex functions. We now proceed to compare the two notions of subgradient employed in Theorem \ref{thm: GR} and Theorem \ref{thm: HB}, respectively. We first would like to show that if $\Phi^*:\mathcal P\To\mathcal L(P)$  is a subgradient of a convex function $\Phi$ on $\mathcal P$, then $S$ in Theorem \ref{thm: GR} extends to a subgradient of $\Phi$ on the positive cone of $\mathcal P$.

\begin{cor}\label{cor: equiv}
Consider a (strictly) convex function $\Phi:\mathcal P\To\R$ that has a subgradient $\Phi^*:\mathcal P\To\mathcal L(P)$ on $\mathcal P$. Then
\begin{equation*}
  S(q)(x) = \Phi^*(q)(x) + \Phi(q) - q\cdot\Phi^*(q)
\end{equation*}
is also a (strict) subgradient of $\Phi$ on $\mathcal P$. Moreover, the 0-homogeneous extension of $S$ is a (strict) subgradient of the 1-homogeneous extension of $\Phi$ on $\Cone \mathcal P$.
\end{cor}
\begin{proof}
The proof follows immediately from Theorem \ref{thm: HB} and the fact that $S$ is a PSR due to Theorem \ref{thm: GR}. However, it would be instructive to show the claim independently. The 0-homogeneous extension of $S$ is given by
\begin{equation}
  S(q)(x) =  \Phi^*\left(\frac q {q\cdot 1}\right)(x) + \Phi\left(\frac q {q\cdot 1}\right) - \frac q {q\cdot 1} \cdot\Phi^*\left(\frac q {q\cdot 1}\right).
\end{equation}
Clearly,
\[
  q\cdot S(q)= (q\cdot 1) \Phi\left(\frac q {q\cdot 1}\right),
\]
for any $q\in\Cone \mathcal P$, as desired. We also have,
\begin{align*}
  p\cdot S(q) &= p\cdot   \Phi^*\left(\frac q {q\cdot 1}\right)(x) + (p\cdot 1) \left(\Phi\left(\frac q {q\cdot 1}\right) - \frac q {q\cdot 1} \cdot\Phi^*\left(\frac q {q\cdot 1}\right)\right)\\
              &\leq \left(\left(\frac  p {p\cdot 1}- \frac q {q\cdot 1}\right)\cdot \Phi^*\left(\frac q {q\cdot 1}\right) + \Phi\left(\frac q {q\cdot 1}\right)\right)(p\cdot 1)\\
              &\leq (p\cdot 1)\Phi\left(\frac p {p\cdot 1}\right),
\end{align*}
for any $p,q\in\Cone \mathcal P$, as desired.
\end{proof}

Another useful consequence of the above characterisations is the following.

\begin{cor}\label{cor: 2}
Let $\Phi:\mathcal P\To\R$ be a (strictly) convex function  that has a subgradient $\Phi^*:\mathcal P\To\mathcal L(P)$ on $\mathcal P$. Then $\Phi^*$ is a (strictly) PSR associated with  $\Phi$ if and only if $q\cdot\Phi^*(q)=\Phi(q)$, for every $q\in\mathcal P$.
\end{cor}
\begin{proof}
The proof follows directly from the hypothesis and \eqref{eq: GR}.
\end{proof}

In the following example, we illustrate how the two theorems may be applied effectively to generate  PSRs from convex functions. The reader may compare our methods of deriving PSRs with those of \cite{D}.

\begin{exmpl}\label{exmpl: char}
Let $\mathcal P$ denote the set of probability densities in the Lebesgue space $L^2(\Omega,\mu)$. We consider the \emph{quadratic entropy} $\Phi(q) = q\cdot q$ on $\mathcal P$ and wish to find a PSR associated with $\Phi$. It is sufficient to find any subgradient of $\Phi$ on $\mathcal P$. The easiest way of doing so is to extend $\Phi$ as the quadratic function on $\Span \mathcal P = L^2(\Omega,\mu)$ and make use of the fact that the extended entropy is differentiable. Its functional derivative is given by
\begin{equation*}
  \Phi^*(q)=2q,
\end{equation*}
which is also a subgradient of $\Phi$ on $L^2(\Omega,\mu)$, and in particular on $\mathcal P$. However, $\Phi^*$ is not a PSR associated with $\Phi$ as $q\cdot\Phi^*(q)\not=\Phi(q)$. For that reason, we apply Theorem \ref{thm: GR} to find that
\begin{equation*}
  S(q) = \Phi^*(q) + \Phi(q) - q\cdot \Phi^*(q) = 2q - q\cdot q
\end{equation*}
is a PSR associated with $\Phi$. This scoring rule is known as the \emph{quadratic scoring rule}.

On the other hand, let us next consider the \emph{spherical entropy} on $\mathcal P$, defined as $\Phi(q) = (q\cdot q)^{1/2}$, and also find a PSR associated with  it. Notice now that $\Phi$ has a natural extension to $\Span\mathcal P$ as the $L^2$-norm, which is a sublinear function. Using the fact that $\Phi$ is a composition of the functions $x\To x^{1/2}$ and $q\To q\cdot q$, we find that its functional derivative on $\Span\mathcal P$ is given by
\begin{equation*}
  \Phi^*(q)= \frac q {(q\cdot q)^{1/2}}.
\end{equation*}
In the light of either Theorem \ref{thm: HB} or Corollary \ref{cor: 2}, $\Phi^*$ is a PSR associated with  $\Phi$. This scoring rule is known as the \emph{spherical scoring rule}.
\end{exmpl}

\section{General convex extensions}\label{sect: conv}

In certain situations, we need to consider more general convex extensions of the entropy function to and beyond the positive cone. For example,  as we saw in Example \ref{exmpl: char}, the simplest convex extension of quadratic entropy to the whole space is the quadratic function $\Phi(q)=q\cdot q$,  while the 1-homogeneous extension of $\Phi$ to $\Cone\mathcal P$,  $\tilde \Phi(q)=q\cdot q/q\cdot1$, cannot be defined for signed densities for which $q\cdot1=0$.

We recall that by $\mathcal K$ we denote a convex set such that $\mathcal P\subset \mathcal K\subset\Span \mathcal P$.

\begin{defn}
Suppose that $\Phi:\mathcal K\To\R$ has a subgradient $\Phi^*:\mathcal K\To\mathcal L(\mathcal P)$ on $\mathcal K$. The \emph{functional Bregman divergence} on $\mathcal K$ associated with the pair $(\Phi,\Phi^*)$ is the function $D_{(\Phi,\Phi^*)}:\mathcal K\times\mathcal K\To\R$ given by
\begin{equation}\label{eq: Div Breg}
  D_{(\Phi,\Phi^*)}(p,q) = \Phi(p) - (p-q) \cdot \Phi^*(q) - \Phi(q),
\end{equation}
for all $p,q\in\mathcal K$.
\end{defn}

We note that $D$ is always nonnegative, while $D$ is positive-definite if and only if $\Phi$ is strictly convex. Notice that if
\[
  S(q)(x) = \Phi^*(q)(x)  + \Phi(q) - q\cdot\Phi^*(q)
\]
is a PSR on $\mathcal P$, then
\begin{equation*}
  p\cdot S(p) - p\cdot S(q)  = D_{(\Phi,\Phi^*)}(p,q)
\end{equation*}
is a Bregman divergence on $\mathcal P$. Hence, score divergences are Bregman divergences for probability densities.

On the extended domain $\mathcal K$, the Bregman divergence is defined as the vertical distance between $\Phi$ and the supporting hyperplanes to $\Phi$ generated by $\Phi^*$. Consider the function $s:\mathcal K\times\mathcal K\To\R$ given by
\begin{equation*}
  s(p,q) = (p-q)\cdot\Phi^*(q) + \Phi(q),
\end{equation*}
for all $p,q\in\mathcal K$, which allows us to write \eqref{eq: Div Breg} simply as
\begin{equation*}
  D_{(\Phi,\Phi^*)}(p,q) = s(p,p)-s(p,q)
\end{equation*}
for all $p,q\in\mathcal K$. Notice that for each $q\in\mathcal K$, $s(\cdot,q)$ is an affine functional on $\Span\mathcal P$.

In order to present the following definition,  we denote by $\mathcal A(\mathcal P)$ the vector space of affine functionals $A$ on $\Span\mathcal P$ of the form $A(p) = p\cdot f + \alpha$, where $f\in \mathcal L(\mathcal P)$ and $\alpha\in\R$ is a constant.

\begin{defn}
Any mapping $S: \mathcal K\To\mathcal A(\mathcal P)$ is said to be an \emph{affine scoring rule} on $\mathcal K$. The associated function $s: \Span \mathcal P\times\mathcal K\To\R$, defined as $s(p,q) = S(q)( p)$, is the \emph{score function} of $S$. The rule $S$ is said to be (strictly) proper if its score function $s$ (strictly) satisfies
\[
  s(p,q)\leq s(p,p)
\]
for all $p,q\in\mathcal K$.
\end{defn}

The following characterisation of proper affine scoring rules is due to \cite{FK}, who refer to affine scoring rules as \emph{affine scores}.

\begin{thm}\label{thm: extend score}
An affine scoring rule $S: \mathcal K\To\mathcal A(\mathcal P)$ is (strictly) proper if and only if there is a (strictly) convex function $\Phi:\mathcal K\To\R$ and a subgradient $\Phi^*:\mathcal K\To\mathcal L(\mathcal P)$ of $\Phi$ on $\mathcal K$ such that
\begin{equation}\label{eq: extend s}
  s(p,q) = (p-q)\cdot\Phi^*(q) + \Phi(q)
\end{equation}
for all $p,q\in\mathcal K$.
\end{thm}


Let us now describe the important special case where an affine scoring rule is in fact linear and may be identified with a family of subgradients of a convex function. To that end, let $\mathcal C$ denote a convex cone such that $\mathcal P\subset\mathcal C\subset\Span \mathcal P$.

\begin{cor}\label{cor: extend score}
Let $S: \mathcal C\To\mathcal A(\mathcal P)$ be a proper affine scoring rule and let $\Phi:\mathcal C\To\R$, $\Phi(p)=s(p,p)$, be the associated extended entropy. Then, $S$ is a linear map if and only if $\Phi$ is 1-homogeneous on $\mathcal C$.
\end{cor}

To summarise, in this section we  have  considered convex extensions of the entropy function outside the set of probabilities $\mathcal P$. Any family of supporting hyperplanes to an extended entropy function defines a proper score function, which generalises the expected score of a PSR. The construction also induces the class of functional Bregman divergences.

\section{Properties of functional Bregman divergences}\label{sect: Bregman}

Here we generalise some basic properties of Bregman divergences to the functional setting. In our first result we characterise functional Bregman divergences under the notion of subgradient that is natural for PSRs. The result extends a similar claim in \cite[Appendix A]{Baner} from the Euclidean setting.

In this section again $\mathcal K$ denotes a convex set such that $\mathcal P\subset \mathcal K\subset\Span \mathcal P$.

\begin{thm}
Let $D:\mathcal K\times\mathcal K\To\R$ be a divergence on $\mathcal K$. Then $D$ is a functional Bregman divergence on $\mathcal K$ if and only if for any $a\in\mathcal K$ the function $\Phi(p)=D(p,a)$ is (strictly) convex and $\Phi$ has a subgradient $\Phi^*:\mathcal K\To\mathcal L(\mathcal P)$ such that
\begin{equation*}
  D(p,q)=D_{(\Phi,\Phi^*)}(p,q)
\end{equation*}
for all $p,q\in\mathcal K$. 
\end{thm}
\begin{proof}
Suppose that $D$ is a functional Bregman divergence associated with some pair $(\Phi_1,\Phi_1^*)$. Then, the function
\[
  \Phi(q)=\Phi_1(p) - p\cdot\Phi_1^*(a) + a\cdot\Phi_1^*(a) - \Phi_1(a)
\]
is (strictly) convex. Since $\Phi(q)$ and $\Phi_1(p)$ only differ by an element in $\mathcal A(\mathcal P)$, they generate the same Bregman divergence. The sufficiency part is trivial.
\end{proof}

A divergence function $D$ on $\mathcal K$ is said to be \emph{symmetric} whenever $D(p,q)=D(q,p)$ for all $p,q\in\mathcal K$. \cite{BauBor, Boi} study the symmetric Bregman divergences on the real line and on Euclidean spaces, respectively. The former authors show that the generalised (or weighted) quadratic divergence is the only symmetric divergence on the real line. We note that the proof easily extends to separable Bregman divergences. Let us recall that a functional Bregman divergence $D:\mathcal K\times\mathcal K\To\R$ is \emph{separable} if $D$ is in the form
\begin{equation*}
   D(p,q) = \int_{\Omega} D_f(p(x),q(x))d\nu(x)
\end{equation*}
for any $p,q\in\mathcal K$, where $D_f$ is a Bregman divergence on the real line induced by some convex differentiable function $f:\R\To\R$, and $\nu$ is a measure on $\Omega$ that is absolutely continuous with respect to $\mu$. In what follows, we present a generalisation of that proof to the context of a large class of non-separable Bregman divergences.

To that end, let $\Phi:\mathcal K\To\R$ be a convex function of the form
\begin{equation}\label{eq: Phi family}
 \Phi(p) =  \phi\left(\int_{\Omega} f(p(x))d\nu(x)\right),
\end{equation}
where $f$ and $\nu$ are as above, while $\phi:\R\To\R$ is an increasing function. This family includes, for example, the pseudospherical scoring rules, whose divergences are evidently non-separable. When $\phi$ is the identity, we recover the class of entropy functions that generate all separable Bregman divergences. Suppose that $\Span\mathcal P$ may be identified with a  Fr\'echet  space $\mathcal N$, and let $\mathcal K$ be an open convex set in $\mathcal N$ containing $\mathcal P$. We denote by $\mathcal N^*$ the topological dual space of $\mathcal N$, which we assume to be identifiable with a subspace of $\mathcal L(\mathcal P)$.

\begin{thm}\label{thm: sym}
Let $\Phi:\mathcal K\To\R$ be a strictly convex function of the form \eqref{eq: Phi family}. Suppose also that $\phi$ and $f$ are twice differentiable and $\Phi$ is twice Fr\'echet differentiable. If the associated functional Bregman divergence is symmetric, then $\Phi$ has the form
\begin{equation*}
  \Phi(q) = \left(\int_{\Omega} q d\nu\right)^2, \quad \text{or} \quad \Phi(q) = \int_{\Omega} q^2 d\nu,
\end{equation*}
up to affine terms $\alpha \int_{\Omega} q d\nu + \beta$, where $\alpha$ and $\beta$ are real constants.
\end{thm}

The proof is relegated to the Appendix. In view of the theorem, the only symmetric functional Bregman divergences on $\mathcal K$ induced by convex functions $\Phi$ in the form \eqref{eq: Phi family} are the following:
\begin{align*}
  D_1(p,q) &= \left(\int_{\Omega} \left(p(x)-q(x)\right)d\nu(x)\right)^2\\
  D_2(p,q) &= \int_{\Omega} (p(x)-q(x))^2d\nu(x).
\end{align*}
Notice that by the Cauchy-Schwartz inequality,
\begin{equation*}
  D_1(p,q) \leq \left(\int_{\Omega} 1d\nu(x) \right)D_2(p,q),
\end{equation*}
hence the quadratic divergence $D_2$ has greater discriminatory power than $D_1$. The two divergences may be regarded as members of the class of generalised quadratic divergences, but we will not try to formalise the latter notion in infinite dimensions. Our negative result means that apart from the generalised quadratic divergences, all other Bregman divergences are nonsymmetric.

For completeness, we note that \cite{Boi} show that if $Q$ is a positive-definite matrix of dimension $n$, then the \emph{generalised quadratic divergence},
\begin{equation*}
  D(p,q) = (p-q)^t Q (p-q),
\end{equation*}
closely related to \emph{Mahalanobis distance}, is the only symmetric Bregman divergence on $\R^n$. Notice that the latter divergence is separable if and only if $Q$ is diagonal. It would be of interest to extend their method of proof to the functional setting, which will likely offer a more general result than Theorem \ref{thm: sym}.

\section{Differentiability properties of entropy functions}\label{sect: diff}

It is well-known that in finite dimensions any convex function on open domain is differentiable and has a unique subgradient everywhere except on a set of Lebesgue measure zero \parencite{Rock}. This implies that an entropy function in finite dimensions determines a unique PSR up to a negligible set. Direct generalisation of this result in infinite dimensions is difficult as there is no analogue of the Lebesgue measure in that setting. Instead, in what follows we describe the subdifferentials of entropy functions and provide sufficient conditions for unique subgradient.

We begin with the case where the extended entropy is a differentiable function with respect to the G\^ateaux derivative, which we review next. To that end, let us suppose that $\Span\mathcal P$ may be identified with a normed space $(\mathcal N, \norm\cdot)$ and denote by $\mathcal N^*$ the topological dual space of $\mathcal N$. Furthermore, let also $\mathcal N^*$ be identifiable with a subspace of $\mathcal L(\mathcal P)$. As usual, the set $\mathcal K$ is convex and $\mathcal P\subset\mathcal K\subset\Span\mathcal P$.

\begin{defn}
Suppose that the set $\mathcal K$ is open with respect to the topology of $\mathcal N$. A function $\Phi:\mathcal K\To\R$ is \emph{G\^ateaux differentiable} at a point $q\in \mathcal K$ if there is $q^*\in \mathcal N^*$ such that for every $p\in \mathcal N$, the limit
\begin{equation*}
  p \cdot q^* = \lim_{t\To0} \frac{\Phi(q+tp)-\Phi(q)}{t}
\end{equation*}
exists. The functional $q^*$ is called the \emph{G\^ateaux derivative} of $\Phi$ at $q$
\end{defn}

The G\^ateaux derivative is necessarily unique from definition. We say that $\Phi$ is differentiable on $\mathcal K$ if $\Phi$ is differentiable at every point in $\mathcal K$. The G\^ateaux derivative has a natural geometric interpretation in the context of convex functions, as shown by the following standard result from convex analysis \parencite{ArtBor, BVff, Zal}.

\begin{thm}\label{thm: Gat}
Suppose that the set $\mathcal K$ is open with respect to the topology of $\mathcal N$, and let $\Phi: \mathcal K\To\R$ be a convex and continuous function. Then, $\Phi$ is G\^ateaux differentiable on $\mathcal K$ if and only if $\Phi$ admits a unique  subgradient $\Phi^*: \mathcal K\To \mathcal N^*$ at each point in $\mathcal K$. In this case $\Phi^*$ is the G\^ateaux derivative of $\Phi$ on $\mathcal K$.
\end{thm}

In the light of the theorem, every convex differentiable function $\Phi:\mathcal K\To\R$ with gradient $\Phi^*:\mathcal K\To\mathcal N^*$ defines a unique collection of supporting hyperplanes to its graph. The restriction to probabilities of these hyperplanes defines the expected score of a unique PSR. We illustrate the theorem with our next example. See also \cite[Section 5]{D} for comparison.



\begin{exmpl}
Let $\mathcal P$ be the set of all probability densities in the Lebesgue space $\mathcal N=L^\gamma(\Omega,\mu)$, for $1<\gamma<\infty$, and consider the power entropy function,
\begin{equation*}
  \Phi_\gamma(p) = \int_\Omega p^{\gamma}(x)d\mu(x),
\end{equation*}
for $p\in L^\gamma(\Omega,\mu)$. We have that $\Span \mathcal P = \mathcal N$ and the topological dual space of $L^\gamma(\Omega,\mu)$ is $\mathcal N^*=L^{\gamma/(\gamma-1)}(\Omega,\mu)$. Clearly, $\mathcal N^*$ may be identified with a subspace of $\mathcal L(\mathcal P)$.

We proceed to compute the G\^ateaux derivative of $\Phi_\gamma$ on $L^\gamma(\Omega,\mu)$. We have
\begin{align*}
  \lim_{t\To0} \frac{\Phi_\gamma(q+tp) - \Phi_\gamma(q)}{t} & = \frac d {dt}\Bigg|_{t=0} (q_t)^\gamma \cdot 1\\
                                            & = p\cdot \gamma q^{\gamma-1}.
\end{align*}
Since $\Phi_\gamma^*(q)=\gamma q^{\gamma-1}\in \mathcal N^*$, $\Phi_\gamma^*$ is indeed the G\^ateaux derivative of $\Phi_\gamma$. Thus, the Bregman divergence on $L^\gamma(\Omega,\mu)$ associated with  $\Phi_\gamma$ is
\begin{equation*}
  D_\gamma(p,q) = p\cdot p^{\gamma-1} - (p-q)\cdot \gamma q^{\gamma-1} - q\cdot q^{\gamma-1}.
\end{equation*}
The associated score function is
\begin{equation*}
  s_\gamma(p,q) = (p-q)\cdot \gamma q^{\gamma-1} + q\cdot q^{\gamma-1},
\end{equation*}
and $\Phi_\gamma(p)=s_\gamma(p,p)$. The restriction of $s_\gamma$ to $\mathcal P$ yields the PSR
\begin{equation*}
  S_\gamma(q) = \gamma q^{\gamma-1} - (\gamma-1)q\cdot q^{\gamma-1},
\end{equation*}
the \emph{power scoring rule} with exponent $\gamma$.
\end{exmpl}

As it is well-known, the above assumptions do not apply to important entropies such as Shannon entropy and Hyv\"arinen entropy, which do not have functional derivatives. On the other hand, all entropy functions of practical interest have well-behaved directional derivatives. Before we recall the relevant definition, we note that in what follows we do not assume that $\Span\mathcal P$ is equipped with topology, and hence $\Span\mathcal P$ is a general vector space. As usual, the set $\mathcal K$ is convex and $\mathcal P\subset\mathcal K \subset \Span \mathcal P$.

\begin{defn}
The \emph{right directional derivative} of $\Phi:\mathcal K\To\R$ at $q\in\mathcal K$ along the vector $p\in\Cone(\mathcal K-q)$ is defined as the limit
\begin{equation}\label{eq: right}
  \Phi_+'(p, q) = \lim_{t\To 0^+} \frac{\Phi(q+t p)-\Phi(q)}{t},
\end{equation}
whenever it exists.
\end{defn}

Geometrically, the set $\Cone(\mathcal K-q)$ gives all non-exterior directions to the set $\mathcal K$ based at $q$. When $\Phi$ is convex, the above limit always exists in a generalised sense that includes convergence to $-\infty$. The subdifferential of $\Phi$ is characterised by the following result.

\begin{thm}\label{thm: exist}
Let $\Phi:\mathcal K\To\R$ be a convex function. Then $\Phi$ has a $\mathcal P$-integrable subgradient at a point $q\in\mathcal K$ if and only if there is $q^*\in\mathcal L(\mathcal P)$ such that
\begin{equation*}
  p\cdot q^* \leq \Phi_+'(p,q)
\end{equation*}
for all $p\in\Cone(\mathcal K-q)$.
\end{thm}

The proof is a minor variant of \cite[Theorem 3.1]{O}.

We next discuss the question when the subdifferential of $\Phi$ at a given point $q\in\mathcal K$ has a unique element in $\mathcal L(\mathcal P)$. First, define the set
\[
  \mathcal O(q)=\Cone(\mathcal K-q)\cap-\Cone(\mathcal K-q),
\]
which is a vector subspace of $\Span\mathcal P$. On this subspace the right directional derivative $\Phi_+'(\cdot,q)$ is always finite \parencite{BVff}. Standardly, if $\mathcal O(q)=\Span\mathcal P$, then $q$ is an (algebraically) interior point of $\mathcal K$ (relative to $\Span\mathcal P$). If $\mathcal K$ has empty interior, however, that is, $\mathcal O(q)\not=\Span\mathcal P$ for any $q\in\mathcal K$, then we may refine the notion of interior by assuming that $\mathcal O(q)$ has in a certain sense negligible complement in $\Span \mathcal P$. We proceed to formalise that sense.

To that end, let us recall that if $E$ is a subset of $\Span\mathcal P$, the set of all $f\in\mathcal L(\mathcal P)$ such that
\begin{equation*}
  p \cdot f = 0,
\end{equation*}
for all $p\in E$, is the \emph{annihilator} of $E$ in $\mathcal L(\mathcal P)$. We denote this set by $E^\perp$. Clearly, $E^\perp$ is a linear subspace of $\mathcal L(\mathcal P)$. In the case where $E^\perp=\{0\}$, we say that $E$ has \emph{trivial annihilator}.

\begin{defn}\label{defn: qint}
Any point $q\in\mathcal K$ such that $\mathcal O(q)$ has trivial annihilator in $\mathcal L(\mathcal P)$ is called an \emph{algebraically quasi-interior} point of $\mathcal K$ relative to $\Span\mathcal P$. The collection of all algebraically quasi-interior points of $\mathcal K$ is the \emph{algebraic quasi-interior} of $\mathcal K$, denoted by $\qint \mathcal K$.
\end{defn}

It is not hard to see that the algebraic quasi-interior of a convex set $\mathcal K$ coincides with the relative interior of $\mathcal K$ in finite dimensions. Similarly, if $\Span\mathcal P$ coincides with a normed space $N$ such that $N^* = \mathcal L(P)$ and $\mathcal K$ has nonempty topological interior, then the topological interior of $\mathcal K$ coincides with the algebraic quasi-interior of $\mathcal K$. If the topological interior of $\mathcal K$ is empty, however, its algebraic quasi-interior may not be empty, which reflects the fact that the spaces $\mathcal O(q)$ do not have to coincide with the whole space $\Span\mathcal P$, as long as their complements are negligible in the precise sense of Definition \ref{defn: qint}. It may also be shown by a standard argument that if $q_1\in\qint\mathcal K$ and $q_2\in\mathcal K$, then the relative interior of the line segment $[q_1,q_2]$ lies in $\qint\mathcal K$. In particular, $\qint\mathcal K$ is convex. Finally, in \cite{O} we show that it is not hard to find convex cones in $L^1(\mathcal \R^n)$ with nonempty quasi-interior that are suitable domains for standard entropy functions such as Shannon entropy and Hyv\"arinen entropy.

\begin{thm}\label{thm: uniq}
Let $\Phi:\mathcal K\To\R$ be a convex function. If $q\in \qint\mathcal K$ and there is $q^*\in\mathcal L(\mathcal P)$ such that
\begin{equation}\label{eq: Phi+ q*}
  p\cdot q^* = \Phi_+'(p,q)
\end{equation}
for all $p\in\Cone(\mathcal K-q)$, then $q^*$ is the unique $\mathcal P$-integrable subgradient of $\Phi$ at $q$ relative to $\mathcal K$.
\end{thm}

If the assumptions of Theorem \ref{thm: uniq} hold for any $q\in\qint\mathcal K$, then the resulting proper affine scoring rule is uniquely associated with $\Phi$ on $\qint \mathcal K$. The proof of the theorem is similar to that of \cite[Theorem 3.2]{O}. See the examples there which show that the logarithmic and  Hyv\"arinen scoring rules are the unique 0-homogeneous $\mathcal P$-integrable subgradients of their entropy functions on the (nonempty) quasi-interior of a suitably chosen positive cone. The theorem is general enough to include all PSRs of probability densities that are of practical interest.

A natural setting to apply the previous two results is the following one. Assume that $\mathcal K$ is large enough so that
\begin{equation}\label{eq: dir P}
  \mathcal P \subset \Cone(\mathcal K-q)
\end{equation}
for any $q\in\mathcal P$.  Condition \eqref{eq: dir P} states that any direction $p\in\mathcal P$ based at any $q\in\mathcal P$ is non-exterior for the set $\mathcal K$ (that is, for some $\lambda>0$, $\lambda p+q\in\mathcal K$). For example, the choice of $\mathcal K = \Cone\mathcal P$ always satisfies condition \eqref{eq: dir P}. Due to \eqref{eq: dir P}, $\Phi_+'(p, q)$ is well-defined for any $p,q\in\mathcal P$ (but may be equal to $-\infty$). If additionally $\mathcal P\subset\qint(\mathcal K)$ and $\Phi$ satisfies the assumptions of Theorem \ref{thm: uniq} for any $q\in\mathcal P$, then $\Phi$ has a unique subgradient at any point in $\mathcal P$ relative to $\mathcal K$ in the class $\mathcal L(\mathcal P)$. The resulting proper scoring rule is uniquely associated with its extended entropy with respect to the latter notion of subgradient.

\appendix

\section*{Appendix}

\begin{proof}[Proof of Theorem \ref{thm: sym}]
Let $\Phi':\mathcal N\To\R$ and $\Phi'':\mathcal N\times\mathcal N\To\R$ denote the first and second Fr\'echet derivatives of $\Phi$ on $\mathcal K$. A computation shows
\begin{align*}
  \xi \cdot \Phi'(p)   =& \phi'\left(\int_{\Omega}f(p(x)d\nu(x)\right) \int_{\Omega}f'(p(x))\xi(x)d\nu(x),\\
  (\xi,\eta)\cdot\Phi''(p)=& \phi'\left(\int_{\Omega}f(p(x))d\nu(x)\right) \int_{\Omega}f''(p(x))\xi(x)\eta(x)d\nu(x) +\\ &\phi''\left(\int_{\Omega}f(p(x))d\nu(x)\right)\int_{\Omega}f'(p(x))\xi(x)d\nu(x)\int_{\Omega}f'(p(x))\eta(x)d\nu(x).
\end{align*}
We remark that ``$\cdot$" denotes both the duality pairing with respect to $\mathcal N$ and $\mathcal N^*$, and with respect to $\Span\mathcal P$ and $\mathcal L(\mathcal P)$. This is well-justified since $f'(p)\rho$ and $f''(p)\rho \xi$ must be in $\mathcal L(\mathcal P)$ for all $p\in\mathcal K$ and all $\xi\in\Span\mathcal P$ due to the hypothesis.

Symmetry of the Bregman divergence associated with $\Phi$ means that we have the identity
\begin{equation*}
  2\Phi(p) - (p-q)\cdot\Phi'(q) = 2\Phi(q) - (q-p)\cdot\Phi'(p)
\end{equation*}
Let $p_t$ denote $p+tr$, for $t\in[0,1]$, $r\in\mathcal N$. Replace $p$ with $p_t$ above and differentiate with respect to $t$ at $t=0$ to find
\begin{equation*}
  2r\cdot\Phi'(p) - r\cdot\Phi'(q) = r\cdot\Phi'(p) -((q-p),r)\cdot\Phi''(p).
\end{equation*}
Since $r\in\mathcal N$ is arbitrary, we have
\begin{equation*}
  \Phi'(q) = (q,\cdot)\cdot\Phi''(p) + \Phi'(p)-(p,\cdot)\cdot\Phi''(p).
\end{equation*}
Fix $p$ and consider that $q$ is the only variable above. Using the explicit form of $\Phi''(p)$, we get that
\begin{equation*}
  \Phi'(q) = 2\alpha aq + 2\beta(q\cdot b) b + c,
\end{equation*}
where $\alpha,\beta\in\R$, and $a,b,c:\Omega\To\R$. In view of the fundamental theorem of calculus for Fr\'echet spaces \parencite[Theorem 3.2.2]{Ham},
\begin{equation*}
  \Phi(q) = \alpha q\cdot aq + \beta (q\cdot b)^2 + q\cdot c + \gamma,
\end{equation*}
where $\gamma$ is a constant of integration. (The latter claim my be verified directly by differentiation.) Since $\Phi$ must be in the form \eqref{eq: Phi family}, the claim follows.
\end{proof}

\printbibliography
\end{document}